\documentclass{zariski}

\title{Differential Geometry of Synthetic Schemes}
\author{Felix Cherubini$^1$, Matthias Ritter$^1$, Hugo Moeneclaey$^1$ and David Wärn$^1$}
\date{$^1$ University of Gothenburg and Chalmers University of Technology }

\begin{document}

\maketitle

\begin{abstract}
  Synthetic algebraic geometry is a new approach to algebraic geometry. It consists in using homotopy type theory extended with three axioms, together with the interpretation of these in a higher version of the Zariski topos, in order to do algebraic geometry internally to this topos. In this article we make no essential use of the higher structure on types, so that we could alternatively use the traditional Zariski 1-topos.
  
  We give new synthetic definitions of étale, smooth and unramified maps between schemes.
  We prove the usual characterizations of these classes of maps in terms of injectivity, surjectivity and bijectivity of differentials. We also show that the tangent spaces of smooth schemes are finite free modules.
  Finally, we show that unramified, étale and smooth schemes can be understood very concretely, indeed they admit the expected local algebraic description.
\end{abstract}

\tableofcontents

\section*{Introduction}

In mathematics, it is common practice to assume a fixed set theory, usually with the axiom of choice, as a common basis. While it is a great advantage to work in one common language and share a lot of the basic constructions, the dual approach of adapting the  ``base language'' to particular mathematical domains is sometimes more concise, provides a new perspective and encourages new proof techniques which would be hard to find otherwise.
We use the word ``synthetic'' to indicate that the latter approach is used,
as it was used by Kock and Lawvere to describe developement of mathematics internal to certain categories \cite{lawvere-categorical-dynamics}, in particular toposes -- a program which dates back as far as 1967.

Already in the 70s in the same program, Anders Kock suggested to use the language of higher-order logic \cite{Church40} to describe the Zariski topos, the collection of sheaves for the Zariski topology \cite{Kock74,kockreyes}, which is the first occurence of synthetic algebraic geometry.
Kock's approach allowed for a more suggestive and geometrical description of schemes.
There is in particular a ``generic local ring'' $R$, which, as a sheaf, associates to any algebra $A$ its underlying set and, as described in \cite{kockreyes}, the projective space $\bP^n$ is then the set of lines in $R^{n+1}$.

Just using category theory is not the same as reasoning synthetically -- for the latter the goal is usually to derive results exclusively in one system,
as Kock and Lawvere did with differential geometry in his work.
The distinction with just using an abstraction like categories is important, since the translation from the synthetic language and back can become cumbersome -- although it is still the goal to derive statements about ordinary mathematical objects in the end.

Starting with Kock and Lawvere's work, more differential geometry was developed synthetically \cite{kock-sdg} along with a study of the models of the theory \cite{moerdijk-reyes}.
One basic axiom of the theory, called the Kock-Lawvere axiom, allows for reasoning with nilpotent infinitesimals. Our version of synthetic algebraic geometry uses a generalisation of this axiom called the duality axiom. Let us now describe the Kock-Lawvere axiom.

The Kock-Lawvere axiom is added to a basic language which can be interpreted in good enough categories, for example toposes. More precisely, we need basic objects like $\emptyset$, $\{\ast\}$ and $\N$ as well as natural constructions like $A\times B$ or $A^B$ for objects $A$, $B$. These constructions come with data, like the projections in the case of $A\times B$, satisfying natural laws. We also need predicates $P(x)$ for elements $x:A$ so we can form subobjects like $\{x:A\mid P(x)\}$.
In this language, we assume there is a fixed ring $R$, which can be thought of as the real numbers. We define $\D(1)=\{x\in R\mid x^2=0\}$ to be the set of all square-zero elements of $R$, then the Kock-Lawvere axiom gives us a bijection
\[ e : R \times R  \to R^{\D(1)} \]
which commutes with evaluation at $0$ and projection to the first factor.
The intuition is that $\D(1)$ is so small that any function on it is linear and therefore determined by its value and its derivative at $0\in\D(1)$.
With this axiom, the derivative at $0:R$ of a function $f : R \to R$ may then be defined as $\pi_2(e^{-1}(f_{\vert \D(1)}))$. This is the start of a convenient development of differential calculus, which doesn't require any further structures on $R$ or other objects. This is the core of the synthetic method: we can work with these differential spaces as if they were sets.

To give an example, the tangent bundle of a manifold $M$ can be defined as $M^{\D(1)}$ and vector fields as sections of the map $M^{\D(1)}\to M$ evaluating at $0$. Then it is easy to see that a vector field is the same as a map $\zeta:\D(1)\to M^M$ with $\zeta(0)=\id_M$, which can be interpreted as an infinitesimal transformation of the identity map. This style of reasoning with spaces as if they were sets is also central in current synthetic algebraic geometry. 

The Kock-Lawvere axiom above are incompatible with the law of excluded middle (LEM) and therefore also with the axiom of choice (AC). Indeed they imply all maps from say $R$ to $R$ are differentiable, which contradicts LEM. 
However restricted versions of LEM and AC are compatible with this axiom. A very basic example is that equality of natural numbers is decidable, meaning that two natural numbers are either equal or not equal. We will latter go back to why full LEM and AC tend to be incompatible with synthetic approach to various geometry.


The use of nilpotent elements to capture infinitesimal quantities as mentioned above was inspired by the Grothendieck school of algebraic geometry and Anders Kock also worked with an extended axiom \cite{Kock74,kockreyes} suitable for synthetic algebraic geometry, where the role of $\D(1)$ above can be taken by any finitely presented affine scheme. In his 2017 doctoral thesis, Ingo Blechschmidt noticed a property holding internally in the Zariski-topos, which he called synthetic quasi-coherence. It is generalised and internalised version of what Kock used. In 2018, David Jaz Myers\footnote{Myers' never published on the subject, but communicated his ideas to Felix Cherubini and in talks to a larger audience \cite{myers-talk1,myers-talk2}.} started working with a specialization of Blechschmidt's synthetic quasi-coherence, which is what we now call \emph{duality axiom}.

To state the duality axiom we need to go from the space $\D(1)$ to spaces that are the common zeros of some finite system of polynomial equations over $R$. Such a space can be encoded independently of the choice of polynomials as a finitely presented $R$-algebra, i.e.\ an $R$-algebra $A$ which is of the form $R[X_1,\dots,X_n]/(P_1,\dots,P_l)$ for some numbers $n,l$ and polynomials $P_i\in R[X_1,\dots,X_n]$.
Then the set of roots of the system is given by the type $\Hom_{\Alg{R}}(A,R)$ of $R$-algebra homomorphisms from $A$ to the base ring. We denote this type by $\Spec A$.
Now the duality axiom states that $\Spec$ is the inverse to exponentiating with $R$, i.e.\ for all 
finitely presented $R$-algebras $A$ the following is an isomorphism:
\[ (a\mapsto (\varphi\mapsto \varphi(a))) : A\to R^{\Spec A}\rlap{.}\]

Myers used homotopy type theory as a base language, which is now the standard in synthetic algebraic geometry. Now we introduce homotopy type theory, in the next paragraphs we will explain how it fits with synthetic algebraic geometry. Homotopy type theory is a language for synthetic homotopy theory.
This means that when using it, we can think of the basic objects of the theory, that is types, directly as homotopy types. This should be contrasted with the usual practice in algebraic topology, which is to implement these homotopy types as topological spaces or Kan complexes.
So the rules of homotopy type theory allow to work with types in very much the same way as one would work with homotopy types in traditional mathematics. 

On the other hand we also use homotopy type theory because it allows to reason synthetically about spaces, as plain type theory does. A key point is that we do not use the law of excluded middle (LEM) or the axiom of choice (AC), which are incompatible with types being interpreted as spaces. Indeed on one hand LEM allows us to find a complement of each subset of a given type $A$, which exposes $A$ as a coproduct.
This is not true for spaces, for example, $\R$ is not the coproduct of the topological subspaces $\{0\}$ and $\R/\{0\}$.
On the other hand AC states that any surjection has a section. This is also not true for any sensible notion of space, in particular it would trivialise all cohomology.
Thus, constructive reasoning in the sense of not using LEM and AC is a necessity if we want to types to be understood as having a spatial structure. It turns out that this is the only obstruction, so the rules of type theory allow to work with type as one would work with spaces in algebraic geometry.

In synthetic algebraic geometry, we work inside homotopy type theory so that types behave both as homotopy types and as spaces from algebraic geometry. This means that we are mixing two synthetic approaches, combining their advantages,
which rests on the possibility of interpreting homotopy type theory in any higher topos \cite{shulman2019all} and not just the higher topos of $\infty$-groupoids. More precisely we think of the higher topos of Zariski sheaves with value in homotopy type. 
The general idea of using homotopy type theory to combine some kind of synthetic, spatial reasoning with synthetic homotopy theory, goes back at least to 2014, to Mike Shulman and Urs Schreiber \cite{Schreiber_2014}.
Schreiber suggested to the HoTT community at various occasions to make use of HoTT as the internal language of higher toposes, where specifities of the topos are accessed in the language via modalities.
This approach was shown to be quite effective and intuitive in Shulman's \cite{shulman-Brouwer-fixed-point} work on mixing synthetic homotopy theory in the form of HoTT and a synthetic approach to topology using a triple of modalities -- a structure called cohesion by Lawvere \cite{Lawvere2007}. A more detailed introduction to homotopy type theory for a general mathematical audience, with an emphasis on this mix of homotopical and spatial structure can be found in \cite{shulman-logic-of-spaces}.

One of the main advantages of using specifically homotopy type theory and not plain type theory, is using synthetic homotopical reasoning to make cohomological computations. Indeed one of Schreiber's motivation was to make use of the modern perspective on cohomology as the connected components of a space of maps in a higher topos. This can be mimicked in HoTT as follows: Given $X$ a type, $A$ an abelian group and $n:\N$, we define the $n$-th cohomology group of $X$ with coefficients in $A$ as
\[ H^n(X,A):=\| X\to K(A,n) \|_0\]
 where $\|\_\|_0$ is the $0$-truncation, an operation which turns any type into a $0$-type, that is a type with trivial higher structure. The type $K(A,n)$ is the $n$-th Eilenberg MacLane space, which can always be constructed for any abelian group $A$ and comes with an isomorphism $\Omega^n(K(A,n))\simeq A$.
With this definition of cohomology groups we can use synthetic homotopy theory to reason about cohomology, which had already been done successfully for the cohomology of homotopy types like spheres and finite cell complexes. It also works for the cohomology of $0$-types such as spaces in synthetic algebraic geometry.
This internal version of cohomology does not agree with the external version mentioned above, indeed it is a sheaf of groups instead of a single group, and it is indexed by an internal natural number instead of an external one. Nevertheless, internal cohomology turned out to be quite useful in practice.

In 2022, trying to use this approach to calculate cohomology groups in synthetic algebraic geometry led to the discovery of what is now called Zariski-local choice \cite{draft},
which is an additional axiom that holds in the higher Zariski-topos.
It is a weakening of the axiom of choice. In homotopy type theory, the axiom of choice can be formulated as follows: For any surjective map $f:X\to Y$, there exists a section, i.e.\ a map $s:Y\to X$ such that $f\circ s=\id_Y$.
Zariski-local choice also states the existence of a section, but only Zariski-locally and only for surjections into an affine scheme: For any surjection $f:E\to \Spec A$,
there exists a Zariski-cover $U_1,\dots,U_n$ of $\Spec A$ and maps $s_i:U_i\to E$ such that $f(s_i(x))=x$ for all $x\in U_i$.

In homotopy type theory, we use the propositional truncation $\|\_\|$ to define surjections and more generally what we mean by ``exists''.
Propositional truncation turns an arbitrary type $A$ into a type $\|A\|$ with the property that $x=y$ for all $x,y:\|A\|$.
Types with this property are called propositions or (-1)-types in homotopy type theory.
Using a univalent universe of types $\mathcal U$ we have that surjection into a type $A$ are the same as type families $F:A\to \mathcal U$, such that we have $\|F(x)\|$ for all $x: A$.
Using type families instead of maps allows us to drop the condition that the maps we get are sections, since we can express it using dependent function types and we arrive at the formulation of Zariski-local choice given below in the list of axioms.
In this instance and many others, homotopy type theory is much more convenient for formal reasoning, which is an advantage when formalizing synthetic algebraic geometry.

In total, the system we use for synthetic algebraic geometry consists of the extension of homotopy type theory postulating a fixed commutative ring $R$ satisfying these three axioms (see below for an explanation of the first one):

\begin{center}
\begin{axiom}[Locality]%
  \label{loc-axiom}
  $R$ is a local ring, i.e.\ $1\neq 0$ and whenever $x+y$ is invertible then $x$ is invertible or $y$ is invertible.
\end{axiom}

\begin{axiom}[Duality]%
  \label{duality-axiom}
  For any finitely presented $R$-algebra $A$, the homomorphism
  \[ a \mapsto (\varphi\mapsto \varphi(a)) : A \to (\Spec A \to R)\]
  is an isomorphism of $R$-algebras.
\end{axiom}

\begin{axiom}[Zariski-local choice]%
  \label{Z-choice-axiom}
  Let $A$ be a finitely presented $R$-algebra
  and let $B : \Spec A \to \mU$ be a family of inhabited types.
  Then there exists a Zariski-cover $U_1,\dots,U_n\subseteq \Spec A$
  together with dependent functions $s_i : (x : U_i)\to B(x)$.
\end{axiom}
\end{center}

As we explained above the duality axiom is a generalisation of the Kock-Lawvere axiom, which was used for convenient infinitesimal computations. It has a lot of consequences. In line with classical algebraic geometry, it shows that we have an anti-equivalence between finitely presented $R$-algebras and affine schemes of finite presentation over $R$.
More surprisingly, it implies that all functions in $R\to R$ are polynomials and that the base ring $R$ has surprisingly strong properties.
For example, for all $x:R$, we have that $x$ is invertible if and only if we have $x\neq 0$.

Surprisingly, the Zariski-local choice axiom was also usable to solve problems which have no obvious connection to cohomology.
For example, it implies that two reasonable definitions of open subsets agree.
In more detail, we can define open subsets using open propositions, which are propositions of the form $r_1\neq 0 \vee\dots\vee r_n\neq 0$ where $r_i:R$.
A subset $U$ of a type $X$ is open if the proposition $x\in U$ is open for all $x:X$.
Given an open subset $U$ of $\Spec A$, using Zariski-local choice we turn these elements $r_1,\dots,r_n$ of the base ring into functions defined Zariski-locally on $\Spec A$.
We can then even prove that $U$ is a union of non-vanishing sets $\mathrm{D}(f_i)$ of global functions $f_i:\Spec A \to R$, which is the second candidate for a definition of open subset alluded to above.
An analogous result holds for closed propositionsand vanishing sets of functions on affine schemes, where closed propositions are propositions of the form $r_1=0\wedge\dots\wedge r_n=0$ where $r_i:R$.

This connection between pointwise and Zariski-local openness is crucial to make the synthetic definition of a scheme work well:
A scheme is a type $X$, that merely has a finite open cover by affine schemes.
To produce interesting examples, it is necessary to use the locality axiom.
This is related to the Zariski topology and ensures that classical examples of Zariski covers can be reproduced.
A central example are the projective spaces $\bP^n$, which can be defined as the quotients of $R^{n+1}/\{0\}$ by the action of $R^\times$ by scaling.
A cover of $\bP^n$ is given by sets of equivalence classes of the form $\{[x_0:\cdots:x_n] \vert x_i\neq 0 \}$, which is clearly open using the pointwise definition.
To see that it is a cover, one has to note that for $x:R^{n+1}$, we have that $x\neq 0$ is equivalent to one of the entries $x_i$ being different from $0$. In synthetic algebraic geometry, this is the case for the base ring $R$ and the proof uses that $R$ is a local ring.

\paragraph{Contribution and organization of the article. }
We give a novel synthetic definition of formally étale, smooth and unramified types, using what we call \emph{closed dense propositions} (\Cref{def-etale-closed-dense}). We then define étale (resp. smooth, unramified) schemes simply as schemes that happens to be formally étale (resp. smooth, unramified) when seen as types. Étale (resp. smooth, unramified) maps between schemes are defined as maps with étale (resp. smooth, unramified) fibers.
This is an instance of a general phenomenon in synthetic reasoning: concepts which are usually defined locally can be defined fiberwise. 

While describing the infinitesimal structure of schemes in section 2, we also point out a curious discovery: there is a duality between finitely presented modules and finitely copresented modules over the internal base ring $R$ (\Cref{dual-of-fcop-fp} and \Cref{double-dual-identity}).
The latter notion of finitely copresented modules is not very prominent in algebra, but appears naturally in the study of tangent spaces of schemes (\Cref{tangent-finite-copresented}).

We show that the new definitions agree with a straightforward translations of the classical concepts (\Cref{connection-to-ega-definition}) and provide some characterizations using tangent spaces:
a map between schemes is unramified if and only if it induces injections on tangent spaces (\Cref{unramified-map-characterisation}), and a map between smooth schemes is étale (resp. smooth) if and only if it induces isomorphisms (resp. surjections) on tangent spaces (\Cref{etale-schemes-iff-local-iso,smooth-schemes-iff-submersion}).

Finally, we show that unramified, étale and smooth schemes can be described very concretely in the expected way, via conditions on the polynomials locally describing such schemes (\Cref{unramified-iff-locally-std-unramified} and \Cref{standard-etale-are-etale,standard-smooth-is-smooth}). An important intermediate result for the characterization of smooth schemes is that their tangent spaces are finite free $R$-modules (\Cref{smooth-have-free-tangent}).

\paragraph{Acknowledgements.}
We thank Thierry Coquand for discussions on the topic and in particular for explaining a proof of \Cref{extend-from-image} to us.
We thank Marc Nieper-Wißkirchen for a discussion which led to the explanation at the beginning of \Cref{remark-sym-dual}.
Work on this article was supported by the ForCUTT project, ERC advanced grant 101053291.

\section{Formally étale, unramified and smooth types}
In this section we will give our new synthetic definitions of formally étale, formally unramified and formally smooth types and maps. It is remarkable that these definitions work for any type or map rather than just scheme and map between them. Here we will derive consequences of these definition applicable for all types, as well as compare these definitions to the traditional ones.

\subsection{Definitions}

In \cite{draft} it is shown that elements of the base ring $R$ are nilpotent if and only if they are not not zero.
Both nilpotency and double negated equality have been used to describe infinitesimals and the following closed dense propositions can be viewed as closed subspaces of the point which are infinitesimally close to being the whole point:

\begin{definition}
A closed proposition is dense\index{closed dense proposition} if it is merely of the form:
\[r_1=0\land\cdots\land r_n=0\]
with $r_1,\cdots,r_n:R$ nilpotent.
\end{definition}

\begin{remark}
A closed proposition $P$ is closed dense if and only if $\neg\neg P$.
\end{remark}

From a traditional perspective, the inclusion $P\subseteq 1$ of a closed dense proposition into the point would be an infinitesimal extension.
In the following, we will use closed dense propositions to define synthetic analogs of notions which are traditionally defined using lifting properties against classes of infinitesimal extensions.
More details on the connection to traditional definitions will be given in \Cref{connection-to-ega-definition}.

\begin{definition}
  \label{def-etale-closed-dense}
A type $X$ is formally étale (resp. formally unramified, formally smooth)\index{formally étale}\index{formally unramified}\index{formally smooth} if for all closed dense proposition $P$ the map:
\[X\to X^P\]
is an equivalence (resp. an embedding, surjective).
\end{definition}

\begin{remark}
The map $X\to X^P$ is an equivalence (resp. an embedding, surjective) if and only if for any map $P\to X$ we have a unique (resp. at most one, merely one) dotted lift in:
\begin{center}
\begin{tikzcd}
P\ar[r]\ar[d] & X\\
1\ar[ru,dashed]& \\
\end{tikzcd}
\end{center}
\end{remark}

\begin{definition}
A map is said to be formally étale (resp. formally unramified, formally smooth) if its fibers are formally étale (resp. formally unramified, formally smooth).
\end{definition}

\begin{remark}
A type (or map) is formally étale if and only if it is formally unramified and formally smooth.
\end{remark}

\begin{lemma}\label{etale-unramified-smooth-square-zero}
A type $X$ is formally étale (resp. formally unramified, formally smooth) if and only if for all $\epsilon:R$ such that $\epsilon^2=0$, the map:
\[X\to X^{\epsilon=0}\]
is an equivalence (resp. an embedding, surjective).
\end{lemma}

\begin{proof}
The direct direction is obvious as $\epsilon=0$ is closed dense when $\epsilon^2=0$.

For the converse, assume $P=\Spec(R/N)$ a closed dense proposition. Then the map $R\to R/N$ with $N$ finitely generated nilpotent ideal can be decomposed as:
\[R\to A_1\to \cdots A_n = R/N\]
where $A_k$ is a quotient of $R$ by a finitely generated nilpotent ideal and:
\[A_k\to A_{k+1}\]
is of the form:
\[A\to A/(a)\]
for some $a:A$ with $a^2=0$.

We write $P_k = \Spec(A_k)$ and:
\[i_k:P_{k+1}\to P_k\] 
so that $\mathrm{fib}_{i_k}(x)$ is $a(x)=0$ where $a(x)^2=0$ holds.

Then by hypothesis we have that for all $k$ and $x:P_{k}$ the map:
\[X\to X^{\mathrm{fib}_{i_k}(x)}\]
is an equivalence (resp. an embedding, surjective). So the map:
\[X^{P_{k}} \to \prod_{x:P_{k}}X^{\mathrm{fib}_{i_k}(x)} = X^{P_{k+1}}\]
is an equivalence (resp. an embedding, surjective, where surjectivity uses $P_{k}$ having choice).
We conclude that the map:
\[X\to X^P\]
is an equivalence (resp. an embedding, surjective).
\end{proof}

\subsection{Stability results}

Being formally étale is a modality given as nullification at all dense closed propostions and therefore lex \cite[Corollary 3.12]{modalities}.
This means that we have the following:

\begin{proposition}
Formally étale types enjoy the following stability results:
\begin{itemize}
\item If $X$ is any type and for all $x:X$ we have $Y_x$ formally étale, then $\prod_{x:X}Y_x$ is formally étale. 
\item  If $X$ is formally étale and for all $x:X$ we have $Y_x$ formally étale, then $\sum_{x:X}Y_x$ is formally étale. 
\item If $X$ is formally étale then for all $x,y : X$ the type $x=y$ is formally étale.
\item The type of formally étale types is itself formally étale.
\end{itemize}
\end{proposition}

Formally unramified type are the separated types \cite[Definition 2.13]{localization} associated to formally étale types. This means:

\begin{lemma}
A type $X$ is formally unramified if and only if for all $x,y:X$ the type $x=y$ is formally étale.
\end{lemma}

By \cite[Lemma 2.15]{localization}, being formally unramified is a nullification modality as well. This means we have the following:

\begin{proposition}
Formally unramified types enjoy the following stability results:
\begin{itemize}
\item If $X$ is any type and for all $x:X$ we have $Y_x$ formally unramified, then $\prod_{x:X}Y_x$ is formally unramified. 
\item If $X$ is formally unramified and for all $x:X$ we have $Y_x$ formally unramified, then $\sum_{x:X}Y_x$ is formally unramified.
\end{itemize}
\end{proposition}

Being formally smooth is not a modality, indeed we will see it is not stable under identity types. Neverthless we have the following results:

\begin{lemma}\label{smooth-sigma-closed}
Formally smooth types enjoy the following stability results:
\begin{itemize}
\item If $X$ is any type satifying choice and for all $x:X$ we have $Y_x$ formally smooth, then $\prod_{x:X}Y_x$ is formally smooth.
\item If $X$ is a formally smooth type and for all $x:X$ we have $Y_x$ formally smooth, then $\sum_{x:X}Y_x$ is formally smooth.
\end{itemize}
\end{lemma}

\subsection{Type-theoretic examples}

The next proposition implies that open propositions, and therefore open embeddings, are formally étale.

\begin{lemma}\label{not-not-stable-prop-etale}
  Any $\neg\neg$-stable proposition is formally étale.
\end{lemma}

\begin{proof}
  Assume $U$ is a $\neg\neg$-stable proposition. For $U$ to be formally étale it is enough to check that $U^P\to U$ for all $P$ closed dense. This holds because for $P$ closed dense we have $\neg\neg P$.
  \end{proof}

Before proving the next lemma about closed formally étale propositions,
we will state and prove a general fact about closed propositions:

\begin{lemma}
  \label{square-zero-implies-zero-decidable}
  Let $I$ be a finitely generated ideal of $R$ such that $I^2=0$ implies $I=0$.
  Then the closed proposition $I=0$ is decidable.
\end{lemma}

\begin{proof}
  Let $I\subseteq R$ be a finitely generated ideal such that $I^2=0$ implies $I=0$.
  Since the other implication always holds, the propositions $I^2=0$ and $I=0$ are equivalent, so we have $I=I^2$.
  By Nakayama (see \cite[Lemma II.4.6]{lombardi-quitte}) there exists $e:R$ such that $eI = 0$ and $1-e\in I$.
  If $e$ is invertible then $I=0$, if $1-e$ in invertible then $I=R$.
\end{proof}

\begin{lemma}\label{closed-and-etale-decidable}
Any formally étale closed proposition is decidable.
\end{lemma} 

\begin{proof}
Given a formally étale closed proposition $P$, let us prove it is $\neg\neg$-stable. Indeed if $\neg\neg P$ then $P$ is closed dense so that $P\to P$ implies $P$ since $P$ is formally étale. 

Let $I$ be the finitely generated ideal in $R$ such that:
\[P\leftrightarrow I=0\]
We have that $I^2=0$ implies $\neg\neg (I=0)$ which implies $I=0$.
Then $P$ is decidable by \Cref{square-zero-implies-zero-decidable}. 
\end{proof}

\begin{proposition}\label{bool-is-etale}
  The type $\Bool$ is formally étale.
\end{proposition}

\begin{proof}
The identity types in $\Bool$ are decidable so $\Bool$ is formally unramified. Consider $\epsilon:R$ such that $\epsilon^2=0$ and a map:
\[\epsilon=0 \to \Bool\]
we want to merely factor it through $1$.

 Since $\Bool\subseteq R$, by duality the map gives $f:R/(\epsilon)$ such that $f^2=f$. Since $R/(\epsilon)$ is local we conclude that $f = 1$ or $f=0$ and so the map has constant value $0:\Bool$ or $1:\Bool$.
\end{proof}

\begin{remark}\label{finite-are-etale}
This means that formally étale (resp. formally unramified, formally smooth) types are stable by finite sums. In particular finite types are formally étale.
\end{remark}

\begin{proposition}
The type $\N$ is formally étale.
\end{proposition}

\begin{proof}
Identity types in $\N$ are decidable so $\N$ is formally unramified, we want to show it is formally smooth. Assume given a map:
\[P\to \N\]
for $P$ a closed dense proposition, we want to show it merely factors through $1$. By boundedness the map merely factors through a finite type, which is formally étale by \Cref{finite-are-etale} so we can conclude.
\end{proof}

\begin{lemma}\label{prop-are-unramified}
Any proposition is formally unramified.
\end{lemma}

This means that any subtype of a formally unramified type is formally unramified.

\begin{remark}
  Given any lex modality, a type is separated if and only if it is a subtype of a modal type,
  so a type is formally unramified if and only if it is a subtype of a formally étale type.
\end{remark}

We also have the following surprising dual result, meaning that any quotient of a formally smooth type is formally smooth:

\begin{proposition}\label{smoothSurjective}
If $X$ is formally smooth and $p:X\twoheadrightarrow Y$ surjective, then $Y$ is formally smooth.
\end{proposition}

\begin{proof}
For any $P$ closed dense and any map $P\to Y$, consider the diagram:
 \begin{center}
      \begin{tikzcd}
      P \ar[rd,dashed]\ar[d]\ar[r]& Y\\
      1 \ar[r,dashed,swap,"x"]& X\ar[u,swap,"p",two heads]
      \end{tikzcd}
    \end{center} 
    By choice for closed propositions we merely get the dotted diagonal, and since $X$ is formally smooth we get the dotted $x$, and then $p(x)$ gives a lift.
\end{proof}

\subsection{Classical definitions, examples and counter-examples}

In this section we will show that our definition of étale, smooth and unramified maps between schemes is equivalent to the obvious internal version of the traditional definition. It is important to keep in mind that our schemes are always locally of finite presentation, so the following definition is sensible:

\begin{definition}
  An étale (resp.\ unramified, smooth)\index{étale}\index{smooth}\index{unramified} scheme is a scheme which is formally étale (resp.\ formally unramified, formally smooth) as a type.
  An étale (resp.\ unramified, smooth) map is a map between schemes which is formally étale (resp.\ formally unramified, formally smooth).
\end{definition}

A criterion very similar to next remark appears as the definition of a formally étale, unramified and smooth maps in \cite[§17]{EGAIV4},
except we restrict to finitely presented algebras and finitely generated ideals, as our schemes are assumed locally of finite presentation, and we ask the lift for smoothness to exists only \emph{Zariski-locally}, as suggested in \cite[\href{https://stacks.math.columbia.edu/tag/02GZ}{Tag 02GZ}]{stacks-project}.  It is not clear if this internal criterion corresponds to the external definitions.

\begin{remark}
  \label{connection-to-ega-definition}
  Let $f:X\to Y$ be a map between schemes.
  Then $f$ is étale (resp.\ unramified, smooth) if and only if there exists exactly one (resp.\ at most one, at least one \emph{Zariski-locally}) dotted lift in all squares of the form:
  \begin{center}
    \begin{tikzcd}
      \Spec(A/N)\ar[r,"t"]\ar[d] & X\ar[d,"f"] \\
      \Spec(A)\ar[ru,dashed]\ar[r,swap,"b"] & Y
    \end{tikzcd}    
  \end{center}
where $A$ is a finitely presented $R$-algebra, $N$ a finitely generated nilpotent ideal and the left map is induced by the quotient map $A\to A/N$. In the smooth case, we believe that it is possible to prove global existence of these lifts using cohmological methods.
\end{remark}

\begin{proof}
  The inclusion of a closed dense proposition $P$ into $1$ is a special case of left map in the remark, so we only need to show that étale,
  smooth and unramified maps satisfy the more general lifting property. For étale and unramified maps, we can just apply the lifting property for closed dense propositions for all points in $\Spec A$. So let $f:X\to Y$ be smooth. Then we merely have a lift for each point in $\Spec(A)$ and can apply Zariski-local choice to get the desired result.
\end{proof}

We conclude this section with a few examples and counter examples.

\begin{lemma}\label{An-is-smooth}
For all $k:\N$, we have that $\A^k$ is smooth.
\end{lemma}

\begin{proof}
  Let $P$ be a closed dense proposition and $N$ a nilpotent, finitely generated ideal such that $P=\Spec(R/N)$.
  Since $\Spec(R[X_1,\dots,X_k])=\A^k$, to prove $\A^k$ smooth we just need to find dotted lifts in:
  \begin{center}
    \begin{tikzcd}
      R/N & R[X_1,\dots,X_k]\ar[l]\ar[ld,dashed]  \\
      R\ar[u] & 
    \end{tikzcd}
  \end{center}
  This is easy using the universal property of $R[X_1,\dots,X_k]$.
\end{proof}

\begin{example}
The affine scheme $\Spec(R[X]/X^2)$ is not smooth.
\end{example}

\begin{proof}
If it were smooth, then for any $\epsilon$ with $\epsilon^3=0$ we would be able to prove $\epsilon^2=0$.
Indeed we would merely have a dotted lift in:
 \begin{center}
      \begin{tikzcd}
        R/(\epsilon^2)& R[X]/(X^2)\ar[l,swap,"\epsilon"]\ar[ld,dashed] \\
        R \ar[u]& 
      \end{tikzcd}
    \end{center}
    that is, an $r:R$ such that $(\epsilon+r\epsilon^2)^2=0$. Then $\epsilon^2=0$.
\end{proof}

\begin{example}
The affine scheme $\Spec(R[X,Y]/XY)$ is not smooth.
\end{example}

\begin{proof}
Again, we assume a lift for any $\epsilon$ with $\epsilon^3=0$:
 \begin{center}
   \begin{tikzcd}
     R/(\epsilon^2) & R[X,Y]/(XY)\ar[l]\ar[ld,dashed] \\
     R\ar[u] & 
   \end{tikzcd}
 \end{center}
 where the top map sends both $X$ and $Y$ to $\epsilon$. Then we have $r,r':R$ such that $(\epsilon+r\epsilon^2)(\epsilon+r'\epsilon^2)=0$ so that $\epsilon^2=0$.
\end{proof}

We will proof a generalization of the following example in \Cref{standard-etale-are-etale}.
The essential step is to improve a zero $g(y)=0$ up to some square-zero $\epsilon$ to an actual zero.

\begin{example}
Let $g$ be a polynomial in $R[X]$ such that for all $x:R$ we have that $g(x)=0$ implies $g'(x)\not=0$. Then $\Spec(R[X]/g)$ is étale.
\end{example}

\subsection{Being formally étale, unramified or smooth is Zariski local}

\begin{lemma}\label{etale-zariski-local}
Let $X$ be a type with $(U_i)_{i:I}$ a finite open cover of $X$. Then $X$ is formally étale (resp. formally unramified, formally smooth) if and only if all the $U_i$ are formally étale (resp. formally unramified, formally smooth).
\end{lemma}

\begin{proof}
First, we show this for formally unramified:
\begin{itemize}
\item Any subtype of a formally unramified type is formally unramified by \Cref{prop-are-unramified}, so $X$ formally unramified implies $U_i$ formally unramified.
\item Conversely, is each $U_i$ is formally unramified, then for all $x,y:X$ we need to prove $x=_Xy$ formally étale. But there exists $i:I$ such that $x\in U_i$ and then:
\[x=_Xy \leftrightarrow \sum_{y\in U_i} x=_{U_i}y \]
which is formally étale because open propositions are formally étale by \Cref{not-not-stable-prop-etale}.
\end{itemize}
Now for formally smooth:
\begin{itemize}
\item Open propositions are formally smooth by \Cref{not-not-stable-prop-etale} so that open subtypes of formally smooth types are formally smooth.
\item Conversely if each $U_i$ is formally smooth then $\Sigma_{i:I}U_i$ is formally smooth by \Cref{finite-are-etale}, so we can conclude that $X$ is formally smooth by applying \Cref{smoothSurjective} to the surjection $\Sigma_{i:I}U_i \twoheadrightarrow X$.
\end{itemize}
The result for formally étale immediately follows.
\end{proof}

\begin{corollary}
For all $k:\N$, the projective space $\bP^k$ is smooth.
\end{corollary}

\begin{proof}
By \Cref{etale-zariski-local} it is enough to check that $\A^k$ is smooth. This is \Cref{An-is-smooth}
\end{proof}

\section{Linear algebra and tangent spaces}
In the context of this article these section contains auxiliary definitions and results, although they can be interesting in their own right. They focus on how tangent (resp. cotangent) spaces work for schemes, in this process we give general results about finitely copresented (resp. finitely presented) modules. Of particular interest is the duality between these two class of modules.

\subsection{Modules and infinitesimal disks}

The most basic infinitesimal schemes are the first order neighbourhoods in affine n-space $R^n$. Their algebra of functions is $R^{n+1}$, which is an instance of the more general construction below.

For any $R$-module $M$, there is an $R$-algebra structure on $R\oplus M$ with multiplication given by:
\[(r,m)(r',m') = (rr',rm'+r'm)\]
Algebras of this form are called \emph{square zero extensions} of $R$, since products of the form $(0,m)(0,n)$ are zero.
By this property, for any $R$-linear map $\varphi:M\to N$ between modules $M,N$, the map $\mathrm{id}\oplus \varphi: R\oplus M\to R\oplus N$ is an $R$-algebra homomorphism. In particular, if $M$ is finitely presented, i.e.\ merely the cokernel of some $p:R^n\to R^m$ then $R\oplus M$ is the cokernel of a map between finitely presented algebras and therefore finitely presented as an algebra. 

\begin{definition}
  \index{$\D(M)$}
Given $M$ a finitely presented $R$-module, we define a finitely presented algebra structure on $R\oplus M$ as above and define:
\[\D(M) = \Spec(R\oplus M)\]
This is a pointed scheme by the first projection which we denote $0$
and the construction is functorial by the discussion above.
\end{definition}

We write $\D(n)$ for $\D(R^n)$ so that for example:
\[\D(1) = \Spec(R[X]/(X^2)) = \{\epsilon:R\ |\ \epsilon^2=0\}\]

\begin{definition}\label{derivation-pointwise}
Assume given $M$ a finitely presented $R$-module and $A$ a finitely presented $R$-algebra with $x:\Spec(A)$. An $M$-derivation at $x$ is a morphism of $R$-modules:
\[d:A\to M\]
such that for all $a,b:A$ we have that:
\[d(ab) = a(x)d(b) + b(x)d(a)\]
\end{definition}

\begin{lemma}\label{tangent-are-derivation}
Assume given $M$ a finitely presented module and $A$ a finitely presented algebra with $x:\Spec(A)$. Pointed maps:
\[\D(M)\to_\pt (\Spec(A),x)\] 
correspond to $M$-derivations at $x$.
\end{lemma}

\begin{proof}
Such a pointed map correponds to an algebra map:
\[f : A\to R\oplus M\]
where the composite with the first projection is $x$. This means that, for some module map $d:A\to M$ we have:
\[f(a) = (a(x),d(a))\]
We can immediately see that $f$ being a map of $R$-algebras is equivalent to $d$ being an $M$-derivation at $x$.
\end{proof}

\begin{lemma}\label{equivalence-module-infinitesimal}
Let $M$, $N$ be finitely presented modules. Then linear maps from $M$ to $N$ correspond to
pointed maps from $\D(N)$ to $\D(M)$. 
\end{lemma}

\begin{proof}
By \Cref{tangent-are-derivation} such a pointed map corresponds to an $N$-derivation at $0:\D(M)$.

Such a derivation is a morphism of modules:
\[d:R\oplus M\to N\]
such that for all $(r,m),(r',m'):R\oplus M$ we have that:
\[d(rr',rm'+r'm) = rd(r',m')+r'd(r,m)\]
This implies $d(r,0) = 0$ for all $r : R$, so such a map is entirely determined by the linear map $m\mapsto d(0,m)$. Conversly given a linear map $f:M\to N$ we can check that $(r,m)\mapsto f(m)$ is such a derivation, giving the correspondence.
\end{proof}

\subsection{Tangent spaces}
  \label{remark-sym-dual}

In traditional algebraic geometry, the tangent sheaf is defined as the dual of the cotangent sheaf which is given Zariski locally by the universal module of derivations. We could copy this approach synthetically, but contrary to the traditional picture, we can also define the tangent bundle directly and dualize to get the usual cotangent bundle. This mismatch with the traditional theory comes from the fact that the traditional dualization of $\mathcal O_X$-module sheaves is not the same as our dualization of $R$-module bundles.

We start with the definition of tangent spaces which is also used in synthetic differential geometry:

\begin{definition}
Let $X$ be a type with $x:X$, then we define the \notion{tangent space} $T_x(X)$ \index{$T_x(X)$} of $X$ at $x$ by:
\[T_x(X) = \{t:\D(1)\to X\ |\ t(0)=x\}\]
\end{definition}

\begin{definition}
Given $f:X\to Y$ and $x:X$ we have a map\index{$df_x$}:
\[df_x : T_x(X)\to T_{f(x)}(Y)\]
induced by post-composition.
\end{definition}

\begin{lemma}\label{An-dimension-n}
For all $x:R^n$ we have $T_x(R^n) = R^n$.
\end{lemma}

\begin{proof}
Since $R^n$ is homogeneous we can assume $x=0$. By \Cref{tangent-are-derivation} we know that $T_0(R^n)$ corresponds to the type of linear maps
\[R[X_1,\cdots,X_n] \to R\]
such that for all $P,Q$ we have:
\[d(PQ) = P(0)dQ + Q(0)dP\]
which is equivalent to $d(1) = 0$ and $d(X_iX_j) = 0$, so any such map is determined by its image on the $X_i$ so it is equivalent to an element of $R^n$.
\end{proof}

\begin{lemma}\label{from-D1-to-D2}
Given a scheme $X$ with $x:X$ and $v,w:T_x(X)$, there exists a unique:
\[\psi_{v,w} : \D(2)\to_\pt X\]
such that for all $\epsilon:\D(1)$ we have that:
\[\psi_{v,w}(\epsilon,0) = v(\epsilon)\]
\[\psi_{v,w}(0,\epsilon) = w(\epsilon)\]
\end{lemma}

\begin{proof}
We can assume $X$ is affine. Then $\D(2)\to_\pt X$ is equivalent to the type of $R^2$-derivations at $x$, but giving an $M\oplus N$-derivation is equivalent to giving an $M$-derivation and an $N$-derivation. Checking the equalities is a routine computation.
\end{proof}

\begin{lemma}
For any scheme $X$ and $x:X$, we have that $T_x(X)$ is a module.
\end{lemma}

\begin{proof}
We could give a conceptual proof similar to \cite[Theorem 4.2.19]{david-orbifolds}. Instead we give a more explicit proof with less technical prerequisites.

We define scalar multiplication by sending $v$ to $t\mapsto v(rt)$. Then for addition of $v,w:T_x(X)$, we define:
\[(v+w)(\epsilon) = \psi_{v,w}(\epsilon,\epsilon)\]
where $\psi_{v,w}$ is defined in \Cref{from-D1-to-D2}. We omit checking that this is a module structure.
\end{proof}

\begin{lemma}
For $f:X\to Y$ a map between schemes, for all $x:X$ the map $df_x$ is a map of $R$-modules.
\end{lemma}

\begin{proof}
Commutation with scalar multiplication is immediate. Commutation with addition comes by applying uniqueness from \Cref{from-D1-to-D2} to get:
\[f\circ \psi_{v,w} = \psi_{f\circ v,f\circ w}\]
\end{proof}

\begin{lemma}\label{kernel-is-tangent-of-fibers}
For any map $f:X\to Y$ and $x:X$, we have that:
\[
\mathrm{Ker}(df_x) = T_{(x,\refl_{f(x)})}(\mathrm{fib}_f(f(x)))
\]
\end{lemma}

\begin{proof}
Indeed we have that $\mathrm{fib}_f(f(x))$ pointed by $(x,\refl_{f(x)})$ is the pullback of:
\[
(X,x) \to (Y,f(x)) \leftarrow (1,*)
\]
in pointed types, and we conclude by exponentiating with $(\mathbb{D}(1),0)$.
\end{proof}

\begin{lemma}\label{tangent-finite-copresented}
Let $X$ be a scheme with $x : X$. Then $T_x(X)$ is a finitely
copresented $R$-module.
\end{lemma}

\begin{proof}
We can assume $X$ affine. For some map $P: R^m\to R^n$ we have $X=\mathrm{fib}_P(0)$.
By applying \Cref{kernel-is-tangent-of-fibers} we know that $T_x(X)$ is the kernel of $dP_x : T_x(R^m)\to T_0(R^n)$ for all $x:X$.
We conclude by \Cref{An-dimension-n}.
\end{proof}

\begin{corollary}
  \label{tangent-bundle-scheme}
  Let $X$ be a scheme, then the tangent bundle $X^{\mathbb{D}(1)}$ is a scheme.
\end{corollary}

\begin{proof}
  We give two proofs, the first uses \Cref{tangent-finite-copresented} and the second is a direct computation:
  \begin{enumerate}[(i)]
  \item Any finitely copresented modules is a scheme, indeed it is the set of common zeros of linear functions between finite free modules.
    So by \Cref{tangent-finite-copresented}, all tangent spaces $T_x(X)$ are schemes and:
    \[
      X^{\mathbb{D}(1)}=\sum_{x:X}T_x(X)
    \]
    is a dependent sum of schemes and therefore a scheme.
  \item Let $X$ be covered by open affine $U_1,\dots,U_n$ then $U_1^{\mathbb{D}(1)},\dots,U_n^{\mathbb{D}(1)}$ is an open cover of $X^{\mathbb{D}(1)}$. Indeed given $f:X^{\mathbb{D}(1)}$, by double negation stability of opens we have that $f\in U_i^{\mathbb{D}(1)}$ if and only if $f(0)\in U_i$.
    So we conclude by showing that for any affine $Y=\Spec R[X_1,\dots,X_n]/(f_1,\dots,f_l)$ the tangent bundle $Y^{\mathbb{D}(1)}$ is affine
    by direct computation:
    \begin{align*}
      Y^{\mathbb{D}(1)}&=\Hom_{\Alg{R}}(R[X_1,\dots,X_n]/(f_1,\dots,f_l), R\oplus \epsilon R) \\
                       &= \{(y_1,\dots,y_n):R\oplus \epsilon R \mid \forall i.\, f_i(y_1,\dots,y_n)=0\} \\
                       &= \{(x_1,\dots,x_n,d_1,\dots,d_n):R^{2n} \mid \forall i.\, f_i(x_1,\dots,x_n)=0\text{ and } \sum_j d_j\frac{\partial f_i}{\partial X_j}(x_1,\dots,x_n) =0\} 
    \end{align*}
  \end{enumerate}
\end{proof}

Now we want to define cotangent spaces.

\begin{definition}
  For $M$ an $R$-module, we denote its dual $\Hom_{R}(M,R)$ by $M^\star$.
\end{definition}

\begin{definition}
For $X$ a type with $x : X$, the \notion{cotangent space} of $X$ at $x$ is the dual $T_x^\star(X)$ of the tangent space $T_x(X)$.
\end{definition}

If $X$ is a scheme, then by \cref{tangent-finite-copresented} the cotangent spaces of $X$ are finitely presented.
We will not use the following definition and remark in the rest of this article, but we included them to show the connection with the traditional theory (see \cite[p. 172]{Hartshorne} or \cite[p. 573]{vakil}).

\begin{definition}
For $A$ an $R$-module, there is a universal derivation $d : A \to \Omega_{A/R}$. Elements of $\Omega_{A/R}$ are called \notion{Kähler differentials}.
\end{definition}

More precisely, $\Omega_{A/R}$ is generated as an $A$-module by symbols
$df$ for $f : A$, subject to the relations $d(r\cdot f) = r \cdot df$ for $r : R$ and
$d(fg) = f \cdot dg + g \cdot df$.
It can be seen that if $A$ is finitely presented as an $R$-algebra,
then $\Omega_{A/R}$ is finitely presented as an $A$-module.
Traditionally, the sheaf corresponding to $\Omega_{A/R}$ is the cotangent bundle.
Synthetically, it is enough to show this pointwise on $\Spec(A)$ by \cite[Theorem 8.2.3]{draft}.
To apply this theorem, we first turn $\Omega_{A/R}$ into an $R$-module bundle on $\Spec(A)$ by sending $x:\Spec A$ to $\Omega_{A/R,x}$ the type of $R$-derivations at $x$, as defined in \Cref{derivation-pointwise}. This agrees with tensoring $\Omega_{A/R}$ with $R$ using the evaluation at $x$, which is the general construction used in \cite[Theorem 8.2.3]{draft}.

\begin{remark}
For all $x:\Spec A$, we have $\Omega_{A/R,x} = T_x^\star(X)$ and therefore $\Omega_{A/R} = \prod_{x:\Spec A} T_x^\star(X)$.
\end{remark}

\begin{proof}
We need to show that for $x : X$,
   we have an isomorphism of $R$-modules:
   \[ T^\star_x(X) = \Omega_{A/R,x}=\Omega_{A/R} \otimes_A R\]
By \Cref{tangent-are-derivation} the tangent space $T_x(X)$ corresponds to derivations
$A \to R$, where the $A$-module structure on $R$ is obtained by evaluating at $x$.
By the universal property of Kähler differentials, these derivations correspond to $A$-module maps $\Omega_{A/R} \to R$, or equivalently to elements in $(\Omega_{A/R}\otimes_AR)^\star$. In \Cref{double-dual-identity}, we will see that $M^{\star\star} = M$ for finitely presented $R$-modules $M$, so we can conclude by dualizing.
\end{proof}

\subsection{Infinitesimal neighbourhoods}

\begin{definition}
Let $X$ be a set with $x:X$. The \notion{first order neighborhood} $N_1(x)$ \index{$N_1(x)$}is defined as the set of $y:X$ such that there exists a finitely generated ideal $I\subseteq R$ with $I^{2}=0$ and:
\[I=0 \to x=y\]
\end{definition}

\begin{lemma}\label{first-order-square-zero}
Assume $x,y:R^n$, then $x\in N_1(y)$ if and only if the ideal generated by the $x_i-y_i$ squares to zero.
\end{lemma}

\begin{proof}
Let us denote $I$ the ideal generated by the $x_i-y_i$ so that $x=y$ if and only if $I=0$. 

If $I^2=0$ then it is clear that $y\in N_1(x)$.

Conversely if $y\in N_1(x)$ then there is $J$ such that $J^2=0$ and $J=0 \to I=0$. Then by duality we have that $I\subset J$ so that $I^2=0$.
\end{proof}

\begin{lemma}\label{first-order-schemes}
Let $X$ be a scheme with $x:X$. Then $N_1(x)$ is an affine scheme. 
\end{lemma}

\begin{proof}
If $x\in U$ open in $X$, we have that $N_1(x)\subset U$ so that we can assume $X$ affine. This means $X$ is a closed subscheme $C\subset R^n$. Then by \Cref{first-order-square-zero}, we have that $N_1(x)$ is the type of $y:R^n$ such that $y\in C$ and for all $i,j$ we have that $(x_i-y_i)(x_j-y_j) = 0$, which is a closed subset of $C$ so it is an affine scheme.
\end{proof}

\begin{definition}
A pointed scheme $(X,*)$ is called a \notion{first order (infinitesimal) disk} if for all $x:X$ we have $x\in N_1(*)$.
\end{definition}

\begin{lemma}
  \label{N1-functor}
  $N_1$ extends to a functor from pointed schemes to first order disks.
\end{lemma}

\begin{proof}
It is clear that $N_1$ is functorial, as it is clear that $y\in N_1(x)$ implies $f(y)\in N_1(f(x))$ from the definition of $N_1$. Now we just need to check that for $X$ a scheme with $x:X$, we have that $(N_1(x),x)$ is a first order disk. Then $N_1(x)$ is a disk by \Cref{first-order-schemes}, and it is clear that the first order neighbourhood of $x$ in $N_1(x))$ is the whole type $N_1(x)$.
\end{proof}

\begin{lemma}\label{disk-are-infinitesimal}
A pointed scheme $(X,*)$ is a first order disk if and only if there exists a finitely presented module $M$ such that:
\[(X,*) = (\D(M),0)\]
\end{lemma}

\begin{proof}
First we check that for all $M$ finitely presented and $y:\D(M)$ we have that $y\in N_1(0)$. Let $m_1,\cdots, m_k$ be generators of $M$, then consider $d:M\to R$ induced by $y$, then $y=0$ if and only if $d=0$ and for all $i,j$ we have that:
\[d(m_i)d(m_j) = 0\]
This means that $I = (d(m_1),\cdots,d(m_k))$ has square $0$ and $I=0$ implies $y=0$ so that $y\in N_1(0)$.

For the converse we assume $X$ a first order disk, by \Cref{first-order-schemes} we have that $X$ is affine and pointed, up to translation we can assume $X$ is a closed subset $X\subset R^n$ pointed by $0$. Since $X$ is a first order disk we have that $X\subset N_1(0)$ and by \Cref{first-order-square-zero} we have $N_1(0) = \D(R^n)$.

This means there is a finitely generated ideal $J$ in $R\oplus R^n$ such that $X=\Spec(R\oplus R^n / J)$.
But $0$ corresponds to the first projection from $R\oplus R^n$, so that $0\in X$ means that if $(x,y)\in J$ then $x=0$, so that $J$ corresponds uniquely to a finitely generated sub-module $K$ of $R^n$ and:
\[X = \Spec(R\oplus (R^n/K)) = \D(R^n/K)\] 
\end{proof}

Now we want to study the duality between finitely presented and finitely copresented modules. While it is clear that the dual of a finite presentation yields a finite copresentation, the reverse is not true in general, but we will show  in \Cref{dual-of-fcop-fp} that it is a consequence of the duality axiom. First we need the following two extension results.

\begin{lemma}
  \label{extend-from-kernel}
  Let $M\subseteq R^n$ be the kernel of a linear map between finite free $R$-modules.
  Then any linear map $M\to R$ can be extended to $R^n$.
\end{lemma}

\begin{proof}
  First note that $M$ is affine of the form $\Spec(R[X_1,\dots,X_n]/(l_1,\dots,l_m))$ with $l_i$ linear.
  Let $L:M\to R$ be linear. Let $P:R^n\to R$ be given by taking a preimage of $L$ under the quotient map $R[X_1,\dots,X_n]\to R[X_1,\dots,X_n]/(l_1,\dots,l_m)$.
  By construction, we have $P_{\vert M}=L$.
  Let $P=\sum_{\sigma:\N^{\{1,\dots,n\}}}a_\sigma X_1^{\sigma(1)}\cdots X_n^{\sigma(n)}$.
  Now we can conclude by showing that the linear part of $P$
  \[
    K\colonequiv \sum_{\sigma:\N^{\{1,\dots,n\}}, \sum \sigma =1}a_\sigma X_1^{\sigma(1)}\cdots X_n^{\sigma(n)}
  \]
  extends $L$ as well, i.e.\ we will see $K_{\vert M}=L$.
  
  For all $x:M$ and $\lambda : R$ we have $L(\lambda x)=\lambda L(x)$ and therefore
  \[
    \sum_{\sigma:\N^{\{1,\dots,n\}}}\lambda^{\sum\sigma} a_\sigma x_1^{\sigma(1)}\cdots x_n^{\sigma(n)}=\lambda \sum_{\sigma:\N^{\{1,\dots,n\}}} a_\sigma x_1^{\sigma(1)}\cdots x_n^{\sigma(n)}
  \]
  By comparing coefficients as polynomials in $\lambda$, we have $\sum_{\sigma:\N^{\{1,\dots,n\}}, \sum \sigma \neq 1}a_\sigma x_1^{\sigma(1)}\cdots x_n^{\sigma(n)}=0$,
  which shows $K_{\vert M}=P_{\vert M}=L$.
\end{proof}

\begin{lemma}
  \label{extend-from-image}
  Let $\varphi:R^n\to R^m$ be $R$-linear, then any linear map $\mathrm{im}(\varphi)\to R$ on the image of $\varphi$ can be extended to $R^m$.
\end{lemma}

\begin{proof}\footnote{This proof is due to Thierry Coquand.}
  Let $(a_{i,j})$ be the coefficients of the matrix representing  $\varphi$ with respect to the standard basis, and let us denote the column $(a_{i,j})_{1\leq i\leq m}$ by $A_j$.
  Then the image of $\varphi$ is generated by these columns:
  \[
    \mathrm{im}(\varphi)=\left\{\Sigma_{j=1}^n x_jA_j\mid \forall_j.\, x_j:R\right\}
  \]
  Let $L:\mathrm{im}(\varphi)\to R$ be $R$-linear and $l_j\colonequiv L(A_j)$.
  Applying $L$ to a general element of $\mathrm{im}(\varphi)$ and using linearity yields the following implication:
  \[
    \sum_{j=1}^n x_jA_j = 0 \Rightarrow \sum_{j=1}^n x_jl_j = 0
  \]
  The left side being 0 means that $m$ linear polynomials $P_i(x_1,\dots,x_n)=\sum_{j=1}^n x_ja_{ij} $ vanish simultaneously.
  Let $Q(x_1,\dots,x_n)$ be the linear polynomial on the right side of the implication.
  Then by duality  the implication induces an inclusion of ideals $(Q)\subseteq (P_1,\dots,P_m)$ in $R[X_1,\dots,X_n]$.
  So there $b_i:R[X_1,\dots,X_n]$ such that
  \[
     Q = \sum_{i=1}^m b_iP_i
   \]
   By comparing coefficients it is clear that the $b_i$ can be chosen to be in $R$, which we now assume.
   
   We define a $R$-linear map $K:R^m\to R$ by $(y_1,\dots,y_m)\mapsto\sum_{i=1}^m b_iy_i$.
   $K$ extends $L$:
   \begin{eqnarray*}
     K\left(\sum_{j=1}^n x_jA_j\right) &=&\sum_{i=1}^m b_i\sum_{j=1}^n x_ja_{ij} \\
     &=& \sum_{i=1}^m b_iP_i(x_1,\dots,x_n) \\
     &=& Q(x_1,\dots,x_n) \\
     &=& \sum_{j=1}^n x_jl_j \\
     &=& L\left(\sum_{j=1}^n x_jA_j\right)
   \end{eqnarray*}
\end{proof}

\begin{lemma}
  \label{dual-of-fcop-fp}
  Let $M$ be finitely copresented, i.e.\ let there be an exact sequence
  \begin{center}
    \begin{tikzcd}
      0\ar[r] & M\ar[r,hook, "\varphi"] & R^n\ar[r,"P"] & R^m
    \end{tikzcd}
  \end{center}
  Then the dual of this sequence is exact as well.
  In particular, $M^\star$ is finitely presented.
\end{lemma}

\begin{proof}
  Surjectivity of $\varphi^\star$ follows from \Cref{extend-from-kernel}.
  Linear maps $R^n\to R$ which vanish on $M$ factor through the image of $P$, so exactness at the middle of the dual sequence follows from \Cref{extend-from-image}.
\end{proof}

\begin{corollary}\label{double-dual-identity}
For any module $M$ finitely presented or finitely copresented, we have that $M^{\star\star}=M$.
\end{corollary}

\begin{lemma}\label{equivalence-modules-disks}
  The functor $M\mapsto \D(M^\star)$ from finitely copresented modules to first order disks is an equivalence with inverse $(X,x)\mapsto T_x(X)$.
\end{lemma}

\begin{proof}
It is fully faithful by \Cref{equivalence-module-infinitesimal} and essentially surjective by \Cref{disk-are-infinitesimal}. To check for the inverse it is enough to check that $T_0(\D(M^\star)) = M$. But by \Cref{equivalence-module-infinitesimal} we have that: 
\[T_0(\D(M^\star)) = \left(\D(1)\to_\pt \D(M^{\star})\right) = M^{\star\star}\] 
and we conclude by \Cref{double-dual-identity}.
\end{proof}

\begin{lemma}\label{duality-infinitesimal-tangent}
Let $X$ be a scheme with $x:X$, then we have that $N_1(x) = \D(T_x(X)^\star)$.
\end{lemma}

\begin{proof}
By \Cref{N1-functor} we have that $(N_1(x),x)$ is a first order disk. By \Cref{equivalence-modules-disks} it is enough to check that $T_x(N_1(x)) = T_0(\D(T_x(X)^\star))$. 

It is immediate that any map $f:\D(1)\to X$ uniquely factors through $N_1(f(0))$ so that $T_x(N_1(x)) = T_x(X)$, and we have that $T_0(\D(T_x(X)^\star)) = T_x(X)$ by \Cref{equivalence-modules-disks}.
\end{proof}

\subsection{Projectivity of finitely copresented modules}

Finitely copresented $R$-modules are projective objects in the category of finitely copresented $R$-modules, which means that all surjections between finitely copresented $R$-modules split.

\begin{lemma}\label{tangent-copresented-modules}
Let $M$ be a finitely copresented module, then we have that $T_0(M) = M$.
\end{lemma}

\begin{proof}
We have that $M$ is the kernel of a linear map $P:R^m\to R^n$. By \Cref{kernel-is-tangent-of-fibers} we have that $T_0(M)$ is the kernel of:
\[dP_0:T_0(R^m)\to T_0(R^n)\]
but by \Cref{An-dimension-n} this is a map from $R^m$ to $R^n$, we omit the verification that $dP_0 = P$.
\end{proof}

\begin{lemma}\label{neighborhood-tangent-correspondence-smooth}
Assume given $M,N$ finitely copresented modules with a map $f:M\to N$. The following are equivalent:
\begin{enumerate}[(i)]
\item $f$ is surjective.
\item $f$ merely has a section.
\item The pointed map $\D(M^\star) \to \D(N^\star)$ corresponding to $f$ merely has a pointed section.
\end{enumerate}
\end{lemma}

\begin{proof}
By \Cref{equivalence-modules-disks} we know that (i) is equivalent to (ii). It is clear that (ii) implies (i).

Let us assume (i) and prove (iii). By \Cref{duality-infinitesimal-tangent} and \Cref{tangent-copresented-modules} we know that $\D(M^*)$ is the first order neighbourhood of $0$ in $M$, so that we have a commutative diagram:
\begin{center}
\begin{tikzcd}
\D(M^\star)\ar[r,hook]\ar[d] & M\ar[d,"f"]\\
\D(N^\star)\ar[r,hook,swap,"i"] & N\\
\end{tikzcd}
\end{center}
Since $\D(N^*)$ has choice and $f$ is surjective there is $g:\D(N^\star)\to M$ such that $f\circ g = i$. We know that $f(g(0))$, so by considering $g'=g-g(0)$ we have that $f\circ g' = i$ and $g'(0)$.
Then we can factor $g'$ through $\D(M^\star)$ as $N_1$ is functorial by \Cref{N1-functor}. This gives us a pointed section of the map $\D(M^\star) \to \D(N^\star)$.
\end{proof}

\begin{corollary}
Any finitely copresented module is projective in the category of finitely copresented modules.
\end{corollary}

There is work in progress \cite{wip-serre-duality} which shows that the finitely copresented modules are also projective in the abelian closure of the finite free modules and finitely presented modules turn out to be injective in this category.

\section{Unramified schemes}
\label{unramified-characterisation}
In this short section we present characterisations of unramified schemes and unramified maps between them. The situation is significantly simpler than with smoothness and étaleness.

\subsection{Unramified schemes}

\begin{lemma}\label{unramified-affine-characterisation}
Let $X$ be an affine scheme, the following are equivalent:
\begin{enumerate}[(i)]
\item $X$ is unramified.
\item Identity types in $X$ are decidable.
\item For all $x:X$, we have that $T_x(X)=0$.
\end{enumerate}
\end{lemma}

\begin{proof}
(i) implies (ii): By \Cref{closed-and-etale-decidable}.

(ii) implies (i): Decidable propositions are formally étale.

(ii) implies (iii): Assume given $x:X$ with $t:T_x(X)$, then for all $\epsilon:\D(1)$ we have $\neg\neg(\epsilon = 0)$ so that we have $\neg\neg (t(\epsilon) = t(0))$ which implies $t(\epsilon) = t(0)$ since equality is assumed decidable. Therefore $t = 0$ in $T_x(X)$.

(iii) implies (i): Indeed given $\epsilon:R$ such that $\epsilon^2=0$, assume $x,y:X$ such that $\epsilon=0 \to x=y$. Then $x\in N_1(y)$ so that by \Cref{duality-infinitesimal-tangent} and $T_y(X)=0$ we conclude $x=y$.
\end{proof}

\begin{corollary}\label{unramified-scheme-characterisation}
Let $X$ be a scheme, the following are equivalent:
\begin{enumerate}[(i)]
\item $X$ is unramified.
\item Identity types in $X$ are open.
\item For all $x:X$, we have that $T_x(X)=0$.
\end{enumerate}
\end{corollary}

\begin{proof}
Assume $(U_i)_{i:I}$ a finite cover of $X$ by affine schemes. By \Cref{etale-zariski-local} we have that $X$ is formally unramified if and only if $U_i$ is formally unramified for all $i:I$.

(ii) implies (i). By \Cref{not-not-stable-prop-etale}.

(i) implies (iii). Indeed for all $x:X$ there exists $i:I$ such that $x\in U_i$, then $T_x(X) = T_x(U_i)$ and $T_x(U_i) = 0$ by \Cref{unramified-affine-characterisation}.

(iii) implies (ii). Assume $x,y:X$, then there exists $i:I$ such that $x\in U_i$ and:
\[x=_Xy \leftrightarrow \Sigma_{y\in U_i} x=_{U_i} y\]
By \Cref{unramified-affine-characterisation} we have that identity types in $U_i$ are decidable, so $x=_Xy$ is open.
\end{proof}

\subsection{Unramified morphisms between schemes}

Now we generalise this to maps between schemes.

\begin{proposition}\label{unramified-map-characterisation}
A map between schemes is unramified if and only if its differentials are injective. 
\end{proposition}

\begin{proof}
The map $df_x$ is injective if and only if its kernel is $0$. By \Cref{kernel-is-tangent-of-fibers}, this means that $df_x$ is injective for all $x:X$ if and only if:
\[
\prod_{x:X}T_{(x,\refl_{f(x)})}(\mathrm{fib}_f(f(x)))=0
\]
On the other hand having fibers with trivial tangent space is equivalent to:
\[
\prod_{y:Y}\prod_{x:X}\prod_{p:f(x)=y} T_{(x,p)}(\mathrm{fib}_f(y)) = 0
\]
Both are equivalent by path elimination on $p$.
\end{proof}

\subsection{Unramified schemes are locally standard}

\begin{definition}
A scheme is called standard unramified if it is of the form:
\[\Spec(R[X_1,\cdots,X_n]/P_1,\cdots,P_k)\]
with $k\geq n$ such that the determinant of:
\[\left( \frac{\partial P_i}{\partial X_j}\right)_{1\leq i,j\leq n}\]
is invertible in $R[X_1,\cdots,X_n]/P_1,\cdots,P_k$.
\end{definition}


\begin{lemma}\label{standard-unramified-is-unramified}
A standard unramfied scheme is indeed unramified.
\end{lemma}

\begin{proof}
Given $X$ standard unramified, for all $x:X$ by \Cref{kernel-is-tangent-of-fibers} we have an exact sequence:
  \begin{center}
    \begin{tikzcd}
      0\ar[r] & T_x(X)\ar[r] & R^n\ar[r,"dP_x"] & R^k
    \end{tikzcd}
  \end{center}
  But since $dP_x$ is represented by the Jacobian matrix $\frac{\partial P_i}{\partial X_j}(x)$, the invertibility condition means $dP_x$ is injective and we can conclude.
\end{proof}

\begin{proposition}
  \label{unramified-iff-locally-std-unramified}
A scheme is unramified if and only if it has a cover by standard unramified schemes.
\end{proposition}

\begin{proof}
By \Cref{etale-zariski-local} and \Cref{standard-unramified-is-unramified}, we get the converse.

  For the direct implication, by \Cref{etale-zariski-local} it is enough to consider an affine scheme: 
  \[
  X=\Spec(R[X_1,\cdots,X_n]/P_1,\cdots,P_k)
  \]
 We reason as in \Cref{standard-unramified-is-unramified} to get that the Jacobian matrix $\left(\frac{\partial P_i}{\partial X_j}(x)\right)$ is invertible for all $x:X$, which means that $n\leq k$ and the Jacobians matrix has an invertible $n$-minor. We cover by principal open according to which $n$-minor is invertible and reorder variables and polynomials to get a cover by pieces if the form:
 \[\Spec(R[X_1,\cdots,X_n,Y]/P_1,\cdots,P_k,1-YQ(X))\]
 such that $\left(\frac{\partial P_i}{\partial X_j}(x)\right)_{1\leq i,j\leq n}$ is invertible for all $x:X$ such that $Q(x)\not=0$. If we reorder the quotienting ideal as $P_1,\hdots,P_n,1-YQ(X),P_{n+1},\hdots,P_k$ we get a standard unramified scheme.
\end{proof}

\section{Smooth and étale schemes}
In this section we will give characterizations similar to Section \ref{unramified-characterisation} for smooth and étale schemes.

\subsection{Smooth and étale maps between schemes}

Smooth maps between with a smooth smooth source are precisely submersions.

\begin{corollary}\label{smooth-schemes-iff-submersion}
Let $f:X\to Y$ be a map between schemes with $X$ smooth. Then the following are equivalent:
\begin{enumerate}[(i)] 
\item The map $f$ is smooth.
\item For all $x:X$, the induced map:
\[df : T_x(X)\to T_{f(x)}(Y)\]
is surjective.
\end{enumerate}
\end{corollary}

\begin{proof}
(i) implies (ii). Assume given $v:T_{f(x)}(Y)$, then for all $t:\D(1)$ we have a map:
\[t=0 \to \fib_f(v(t))\]
with constant value $x$. So since $f$ is smooth we merely have $w_t:\fib_f(v(t))$ such that $t=0$ implies $w_t=x$. We conclude using choice over $\D(1)$.

(ii) implies (i). Assume given $y:Y$ and $\epsilon:R$ such that $\epsilon^2=0$ and try to merely find a dotted lift in:
 \begin{center}
      \begin{tikzcd}
        \epsilon=0\ar[r,"\phi"]\ar[d] & \fib_f(y)\\
       1 \ar[ru,dashed] & \\
      \end{tikzcd}
    \end{center}
    Since $X$ is formally smooth we merely have an $x:X$ such that:
\[\prod_{p:\epsilon=0} \phi(p)=x\]
and therefore:
\[ \epsilon=0 \to y=f(x)\]
which means that $y\in N_1(f(x))$. 

We use \Cref{duality-infinitesimal-tangent} and \Cref{neighborhood-tangent-correspondence-smooth} to get that the map $N_1(x)\to N_1(f(x))$ induced by $f$ merely has a section $s$ sending $f(x)$ to $x$. Then $s(y):\fib_f(y)$ is such that for all $p:\epsilon=0$ we have that:
\[\phi(p) = x = s(f(x)) = s(y)\]
\end{proof}

\begin{corollary}\label{etale-schemes-iff-local-iso}
Let $f:X\to Y$ be a map between schemes with $X$ smooth. Then the following are equivalent:
\begin{enumerate}[(i)]
\item The map $f$ is étale. 
\item For all $x:X$, the induced map:
\[df : T_x(X)\to T_{f(x)}(Y)\]
is an iso.
\end{enumerate}
\end{corollary}

\begin{proof}
We apply \Cref{unramified-map-characterisation} and \Cref{smooth-schemes-iff-submersion}.
\end{proof}

\begin{remark}\label{smooth-maps-are-submersions}
In both previous results, we did not use the smoothness hypothesis on $X$ to prove (i) implies (ii).
\end{remark}

\subsection{Smooth schemes have free tangent spaces}

\begin{lemma}\label{smooth-implies-smooth-tangent}
Assume $X$ is a smooth scheme. Then for any $x:X$ the type $T_x(X)$ is smooth.
\end{lemma}

\begin{proof}
Consider $T(X) = X^{\mathbb{D}(1)}$ the tangent bundle of $X$. We have to prove that the map:
\[p:T(X)\to X\]
is formally smooth. Both source and target are schemes, and the source is formally smooth because $X$ is smooth and $\mathbb{D}(1)$ has choice. So by \Cref{smooth-schemes-iff-submersion} it is enough to prove that for all $x:X$ and $v:T_x(X)$ the induced map:
\[dp:T_{(x,v)}(T(X))\to T_x(X)\]
is surjective. 

Consider $u:T_x(X)$. By unpacking the definition of tangent spaces and computing $dp(w)$, we see that merely finding $w:T_{(x,v)}(T(X))$ such that $dp(w) = u$ means merely finding:
\[\phi : \mathbb{D}(1) \times \mathbb{D}(1) \to X\]
such that for all $t:\mathbb{D}(1)$ we have that:
\[\phi(0,t) = v(t)\]
\[\phi(t,0) = u(t)\]

But from \Cref{from-D1-to-D2} we know that there exists a unique:
\[\psi_{v,u} : \mathbb{D}(2)\to X\]
such that:
\[\psi_{v,u}(0,t) = v(t)\]
\[\psi_{v,u}(t,0) = u(t)\]

Then since $X$ is smooth and the fibers of:
\[\mathbb{D}(2) \to\mathbb{D}(1) \times \mathbb{D}(1) \]
are closed dense, we conclude from $\mathbb{D}(1) \times \mathbb{D}(1)$ having choice that there merely exists a lift of $\psi_{v,u}$ to $\mathbb{D}(1) \times \mathbb{D}(1)$, which gives us the $\phi$ we wanted.
\end{proof}

\begin{lemma}\label{smooth-kernel-decidable}
Assume given a linear map:
\[M:R^m\to R^n\] 
which has smooth kernel $K$. Then we can decide whether $M=0$.
\end{lemma}

\begin{proof}
Since $M=0$ is closed, by \Cref{not-not-stable-prop-etale} and \Cref{closed-and-etale-decidable} it is enough to prove that it is $\neg\neg$-stable to conclude that it is decidable. Assume $\neg\neg(M=0)$, then for any $x:R^m$ we have a dotted lift in:
 \begin{center}
      \begin{tikzcd}
        M=0\ar[d] \ar[r,"\_\, \mapsto x"] & K \\
       1 \ar[dashed,ru] &
      \end{tikzcd}
\end{center}
because $K$ is formally smooth, so that we merely have $y\in K$ such that: 
\[M=0\to x=y\]
which implies that $\neg\neg(x=y)$ since we assumed $\neg\neg(M=0)$.

Then considering a basis $(x_1,\cdots,x_n)$ of $R^m$, we get $(y_1,\cdots,y_n)$ such that for all $i$ we have that $y_i\in K$ and $\neg\neg(y_i=x_i)$. But then we have that $(y_1,\cdots,y_n)$ is infinitesimally close to a basis and that being a basis is an open proposition, so that $(y_1,\cdots,y_n)$ is a basis and $K=R^m$ so $M=0$.
\end{proof}

\begin{lemma}\label{smooth-corpresented-implies-free}
Assume that $K$ is a finitely copresented module that is also smooth. Then it is finite free.
\end{lemma}

\begin{proof}
Assume a finite copresentation:
\[0\to K\to R^m\overset{M}{\to} R^n\]
We proceed by induction on $m$. By \Cref{smooth-kernel-decidable} we can decide whether $M=0$ or not.
\begin{itemize}
\item If $M=0$ then $K=R^m$ and we can conclude.
\item If $M\not=0$ then we can find a non-zero coefficient in the matrix corresponding to $M$, and so up to base change it is of the form:
\[
\begin{pmatrix}
1 & \begin{matrix}0&\cdots & 0\end{matrix}  \\
\begin{matrix}0\\ \vdots\\ 0\end{matrix} & \widetilde{M} \\
\end{pmatrix}
\]
But then we know that the kernel of $M$ is isomorphic to the kernel of $\widetilde{M}$, and by applying the induction hypothesis we can conclude that it is finite free.
\end{itemize}
\end{proof}

\begin{proposition}\label{smooth-have-free-tangent}
Let $X$ be a smooth scheme. Then for any $x:X$ we have that $T_x(X)$ is finite free.
\end{proposition}

\begin{proof}
By \Cref{smooth-implies-smooth-tangent} we have that $T_x(X)$ is formally smooth, so that we can conclude by \Cref{smooth-corpresented-implies-free}.
\end{proof}

The dimension of $T_x(X)$ is called the dimension of $X$ at $x$. By boundedness any smooth scheme is a finite sum of smooth schemes of a fixed dimension.
We can turn this into a definition of dimension which works well in the case of smooth schemes:

\begin{definition}
  \label{definition-smooth-dim-n}
  A scheme is \notion{smooth of dimension $n$}, if it is smooth and all its tangent spaces are finite free of dimension $n$.
\end{definition}

\subsection{Standard étale and standard smooth schemes}

\begin{definition}
A standard smooth scheme of dimension $k$ is an affine scheme of the form:
\[\Spec\big(R[X_1,\cdots,X_n,Y_1,\cdots Y_{k}] / P_1,\cdots,P_n\big)\]
where the determinant of:
\[\left( \frac{\partial P_i}{\partial X_j}\right)_{1\leq i,j\leq n}\]
is invertible in $R[X_1,\cdots,X_n,Y_1,\cdots Y_{k}] / P_1,\cdots,P_n$.
\end{definition}

\begin{definition}
A standard smooth scheme of dimension $0$ is called a \notion{standard étale scheme}.
\end{definition}


\begin{lemma}\label{standard-etale-are-etale}
Standard étale schemes are étale.
\end{lemma}

\begin{proof}
Assume given a standard étale algebra:
\[R[X_1,\cdots,X_n]/P_1,\cdots,P_n\]
and write:
\[P:R^n\to R^n\]
for the map induced by $P_1,\cdots,P_n$.

Assume given $\epsilon:R$ such that $\epsilon^2=0$, we need to prove that there is a unique dotted lifting in:
  \begin{center}
      \begin{tikzcd}
       R/\epsilon & R[X_1,\cdots,X_n]/P_1,\cdots,P_n\ar[l,swap,"x"]\ar[dashed,ld] \\
       R\ar[u]&
      \end{tikzcd}
    \end{center}
This means that for all $x:R^n$ such that $P(x)=0$ mod $\epsilon$, there exists a unique $y:R^n$ such that:
\begin{itemize} 
\item We have $x=y$ mod $\epsilon$.
\item We have $P(y)=0$.
\end{itemize}

First we prove existence. For any $b:R^n$ we compute:
\[P(x+\epsilon b) = P(x) + \epsilon\ dP_x(b)\]
We have that $P(x)=0$ mod $\epsilon$, say $P(x) = \epsilon a$. Since $\neg\neg(P(x) = 0)$, we have that $dP_x$ is invertible. Then taking $b = -(dP_x)^{-1}(a)$ gives a lift $y=x+\epsilon b$ such that $P(y) = 0$.

Now we check unicity. Assume $y,y'$ two such lifts, then $y=y'$ mod $\epsilon$ and we have:
\[P(y) = P(y') + dP_{y'}(y-y')\]
and $P(y)=0$ and $P(y')=0$ so that:
\[dP_{y'}(y-y') = 0\]
But $dP_{y'}$ is invertible and we can conclude that $y=y'$.
\end{proof}

\begin{lemma}\label{standard-smooth-is-smooth}
Any standard smooth scheme of dimension $k$ is smooth of dimension $k$. 
\end{lemma}

\begin{proof}
The fibers of the map:
\[\Spec\big(R[X_1,\cdots,X_n,Y_1,\cdots Y_{k}] / P_1,\cdots,P_n\big) \to \Spec(R[Y_1,\cdots Y_{k}])\]
are standard étale, so the map is étale by \Cref{standard-etale-are-etale}. Since:
\[\Spec(R[Y_1,\cdots Y_{k}]) = \A^k\]
is smooth by \Cref{An-is-smooth}, we can conclude it is smooth using \Cref{smooth-sigma-closed}. 

For the dimension we use \Cref{An-dimension-n} and \Cref{smooth-maps-are-submersions}.
\end{proof}

\subsection{Smooth schemes are locally standard smooth}

\begin{proposition}\label{smooth-are-locally-standard}
A scheme is smooth of dimension $k$ if and only if it has a finite open cover by standard smooth schemes of dimension $k$.
\end{proposition}

\begin{proof}
By \Cref{etale-zariski-local} and \Cref{standard-smooth-is-smooth}, we get the converse. 

For the direct implication, by \Cref{etale-zariski-local} it is enough to consider an affine scheme: 
\[
X = \Spec(R[X_1,\cdots,X_m]/P_1,\cdots,P_l)
\]

From \Cref{smooth-have-free-tangent}we get that for any $x:X$ we have that $dP_x$ has finite free kernel of rank $k$. Then by \Cref{rank-equivalent-definitions} we know that $dP_x$ has rank $n=m-k$ for every $x$.

We cover $X$ by principal opens according to which $n$-minor is invertible, so that up to a rearranging of variables and polynomials we can assume a polynomial $G$ such that:
\[X = \Spec(R[X_1,\cdots,X_{n},Z,Y_1,\cdots,Y_k]/P_1,\cdots,P_{n},1-ZG,Q_1,\cdots, Q_p)\]
which can be rewritten as:
\[X = \Spec(R[X_1,\cdots,X_{q},Y_1,\cdots,Y_k]/P_1,\cdots,P_{q},Q_1,\cdots, Q_p)\]
where $q=n+1$ and $\left(\frac{\partial P_i}{\partial X_j}\right)_{i,j}$ is invertible modulo $P_1,\cdots,P_{q},Q_1,\cdots, Q_p$.

 Then we have:
\[d(P,Q)_{x,y} = \begin{pmatrix}
\left(\frac{\partial P}{\partial X}\right)_{x,y} & \left(\frac{\partial P}{\partial Y}\right)_{x,y} \\
\left(\frac{\partial Q}{\partial X}\right)_{x,y} & \left(\frac{\partial Q}{\partial Y}\right)_{x,y} \\
\end{pmatrix}\]
where we used the notation:
\[\left(\frac{\partial P}{\partial X}\right)_{x,y} = \left(\frac{\partial P_i}{\partial X_j}(x,y)\right)_{i,j}\]
so that $\left(\frac{\partial P}{\partial X}\right)_{x,y}$ is invertible of size $q$. Moreover by \Cref{rank-bloc-matrix} we get:
\[\left(\frac{\partial Q}{\partial Y}\right)_{x,y} = \left(\frac{\partial Q}{\partial X}\right)_{x,y}\left(\frac{\partial P}{\partial X}\right)_{x,y}^{-1} \left(\frac{\partial P}{\partial Y}\right)_{x,y} \]
which will be useful later.

Now we prove that for any $(x,y):R^{q+k}$ such that $P(x,y)=0$, the proposition $Q(x,y)=0$ is decidable.

Using \Cref{square-zero-implies-zero-decidable} this follow from:
\[(Q_1(x,y),\cdots,Q_p(x,y))^2=0 \to (Q_1(x,y),\cdots,Q_p(x,y))=0\]
Assuming $(Q_1(x,y),\cdots,Q_p(x,y))^2=0$, by smoothness there is a dotted lift in:
\begin{center}
      \begin{tikzcd}
        R/(Q_1(x,y),\cdots,Q_l(x,y)) & \Spec(R[X_1,\cdots,X_q,Y_1,\cdots,Y_k]/P_1,\cdots,P_q,Q_1,\cdots, Q_l)\ar[l,swap,"(x{,}y)"] \ar[dashed,ld,"(x{'}{,}y{'})"]\\
       R\ar[u] & \\
      \end{tikzcd}
\end{center}
Let us prove that $Q(x,y) = 0$. Indeed we have $(x,y) \sim_1 (x',y')$ so that we have:
\[P(x,y) = P(x',y')+ \left(\frac{\partial P}{\partial X}\right)_{x',y'}(x-x') + \left(\frac{\partial P}{\partial Y}\right)_{x',y'}(y-y') \]
\[Q(x,y) = Q(x',y')+ \left(\frac{\partial Q}{\partial X}\right)_{x',y'}(x-x') + \left(\frac{\partial Q}{\partial Y}\right)_{x',y'}(y-y') \]
Then we have $P(x,y) = 0$, $P(x',y')=0$ and $Q(x',y') = 0$. From the first equality we get:
\[x-x' =  -\left(\frac{\partial P}{\partial X}\right)_{x',y'}^{-1}\left(\frac{\partial P}{\partial Y}\right)_{x',y'}(y-y')\]
so that from the second we get:
\[Q(x,y) = -\left(\frac{\partial Q}{\partial X}\right)_{x',y'}\left(\frac{\partial P}{\partial X}\right)_{x',y'}^{-1}\left(\frac{\partial P}{\partial Y}\right)_{x',y'}(y-y') + \left(\frac{\partial Q}{\partial Y}\right)_{x',y'}(y-y')\]
so that $Q(x,y)=0$ as we have seen previously that:
\[\left(\frac{\partial Q}{\partial Y}\right)_{x',y'} = \left(\frac{\partial Q}{\partial X}\right)_{x',y'}\left(\frac{\partial P}{\partial X}\right)_{x',y'}^{-1} \left(\frac{\partial P}{\partial Y}\right)_{x',y'} \]
From the decidability of $Q(x,y)=0$ we get that $X$ is an open in $\Spec(R[X_1,\cdots,X_q,Y_1,\cdots,Y_k]/P_1,\cdots,P_q)$
so it is of the form $D(G_1,\cdots,G_r)$, and we have an open cover of our scheme by pieces of the form:
\[\Spec(R[X_1,\cdots,X_q,Z,Y_1,\cdots,Y_k]/P_1,\cdots,P_q,1-ZG_i)\]
which is standard smooth of dimension $k$.
\end{proof}

\begin{corollary}
A scheme is étale if and only if it has a cover by standard étale schemes.
\end{corollary}

\begin{proof}
By \Cref{unramified-scheme-characterisation} we know that a scheme is formally étale if and only if it is smooth of dimension $0$. Then we just apply \Cref{smooth-are-locally-standard}.
\end{proof}


\pagebreak
\appendix
\section{Appendix}
\subsection{Rank of matrices}

\begin{definition}
A matrix is said to have rank $\leq n$ if all its $n+1$-minors are zero. It is said to have rank $n$ if it has rank $\leq n$ and does not have rank $\leq n-1$.
\end{definition}

Having a rank is a property of matrices, as a rank function defined on all matrices would allow to e.g. decide if an $r:R$ is invertible.

\begin{lemma}\label{rank-bloc-matrix}
Assume given a matrix $M$ of rank $n$ decomposed into blocks:
\[M = \begin{pmatrix}
P & Q  \\
R & S \\
\end{pmatrix}\]
Such that $P$ is square of size $n$ and invertible. Then we have:
\[S = RP^{-1}Q\]
\end{lemma}

\begin{proof}
By columns manipulation the matrix is equivalent to:
\[M = \begin{pmatrix}
P & Q  \\
0 & S - RP^{-1}Q \\
\end{pmatrix}\]
but equivalent matrices have the same rank so $S=RP^{-1}Q$.
\end{proof}

\begin{lemma}\label{rank-equivalent-definitions}
If a linear map $R^m \to R^n$ given by multiplication with $M$
has finite free kernel of rank $k$, then $M$ has rank $m-k$.
\end{lemma}

\begin{proof}
  Let $a_1,\dots,a_{k}$ be a basis for the kernel of $M$ in $R^m$, which we complete into a basis of $R^m$ via $b_{k+1},\dots,b_m$.
  By completing $Mb_{k+1},\dots, Mb_m$ to a basis of $R^n$, we get a basis where $M$ is written as:
\[\begin{pmatrix}
I_{m-k} & 0  \\
0 & 0 \\
\end{pmatrix}\]
so that $M$ has rank $m-k$.
\end{proof}





\printbibliography

\end{document}